 \documentclass[final,1p,times]{elsarticle}
\usepackage{amssymb}
\usepackage{amsthm}

\usepackage{graphicx}
\usepackage{subcaption}
\usepackage{caption}
\usepackage{float}
\usepackage{mathtools}
\usepackage{amssymb}
\usepackage{todonotes}
\usepackage{verbatim}
\usepackage{bbm}
\usepackage{enumitem}
\usepackage{multirow}

\usepackage{tikz}
\usetikzlibrary{patterns}
\usepackage{pgfplots}
\usepgfplotslibrary{colorbrewer}
\pgfplotsset{compat=1.8}
\definecolor{clr1}{RGB}{27,158,119}
\definecolor{clr2}{RGB}{217,95,2}
\definecolor{clr3}{RGB}{117,112,179}
\definecolor{clr4}{RGB}{231,41,138}
\definecolor{clr5}{RGB}{102,166,30}
\definecolor{clr6}{RGB}{230,171,2}
\definecolor{clr7}{RGB}{166,118,29}
\pgfplotsset{
    cycle list={clr1,clr2,clr3,clr4,clr5,clr6,clr7},
}

\newtheorem{property}{Property}
\newtheorem{remark}{Remark}
\newtheorem{theorem}{Theorem}
\newtheorem{conjecture}{Conjecture}
\newtheorem{lemma}{Lemma}
\newtheorem{example}{Example}
\newtheorem{definition}{Definition}

\journal{ValueTools 2020}

\begin{document}

\begin{frontmatter}

%% Title, authors and addresses

%% use the tnoteref command within \title for footnotes;
%% use the tnotetext command for theassociated footnote;
%% use the fnref command within \author or \address for footnotes;
%% use the fntext command for theassociated footnote;
%% use the corref command within \author for corresponding author footnotes;
%% use the cortext command for theassociated footnote;
%% use the ead command for the email address,
%% and the form \ead[url] for the home page:
%% \title{Title\tnoteref{label1}}
%% \tnotetext[label1]{}
%% \author{Name\corref{cor1}\fnref{label2}}
%% \ead{email address}
%% \ead[url]{home page}
%% \fntext[label2]{}
%% \cortext[cor1]{}
%% \address{Address\fnref{label3}}
%% \fntext[label3]{}

\title{Stability of Redundancy Systems with Processor Sharing}

%% use optional labels to link authors explicitly to addresses:
%% \author[label1,label2]{}
%% \address[label1]{}
%% \address[label2]{}

\author[label1]{Youri Raaijmakers\corref{cor1}}
\ead{y.raaijmakers@tue.nl}
\author[label1]{Sem Borst}
\author[label1]{Onno Boxma}

\address[label1]{Department of Mathematics and Computer Science, Eindhoven University of Technology, 5600 MB Eindhoven, The Netherlands}
\cortext[cor1]{Corresponding author}

\begin{abstract}
We investigate the stability condition for redundancy-$d$ systems where each of the servers follows a processor-sharing (PS) discipline. We allow for generally distributed job sizes, with possible dependence among the $d$ replica sizes being governed by an arbitrary joint distribution. We establish that the stability condition for the associated fluid-limit model is characterized by the expectation of the minimum of $d$ replica sizes being less than the mean interarrival time per server. In the special case of identical replicas, the stability condition is insensitive to the job size distribution given its mean, and the stability condition is inversely proportional to the number of replicas. In the special case of i.i.d.\ replicas, the stability threshold decreases (increases) in the number of replicas for job size distributions that are NBU (NWU). We also discuss extensions to scenarios with heterogeneous servers.
\end{abstract}

\begin{keyword}
%% keywords here, in the form: keyword \sep keyword
Parallel-server system \sep redundancy \sep stability \sep processor-sharing
%% PACS codes here, in the form: \PACS code \sep code

%% MSC codes here, in the form: \MSC code \sep code
%% or \MSC[2008] code \sep code (2000 is the default)

\end{keyword}

\end{frontmatter}

%% \linenumbers

%% main text
\section{Introduction}
\label{sec: introduction}
The interest in redundancy systems has strongly grown in recent years, fueled by empirical evidence that redundancy improves performance in applications with many servers such as web page downloads and Google search queries \cite{AGSS-ESM,VGMSRS-LLR}. The most important feature of redundancy is the replication of each incoming job. The replicas are instantaneously allocated to, say, $d$ different servers, chosen uniformly at random (without replacement), and abandoned as soon as either the first of these $d$ replicas starts service (`cancel-on-start' c.o.s.) or the first of these $d$ replicas finishes service (`cancel-on-completion' c.o.c.). 

As alluded to above, allocating replicas of the same job to multiple servers has the potential to improve delay performance. Indeed, adding replicas increases the chance for one of the replicas to find a short queue, and the c.o.s.\ version is in fact equivalent to a Join-the-Smallest-Workload (JSW) like policy. On the other hand, it may also result in potential vulnerabilities and could cause instability in the c.o.c.\ version since the same job may be in service at multiple servers, potentially wasting service capacity. 
The above trade-off already suggests that establishing the stability condition is not straightforward. 
Indeed, despite the numerous studies on redundancy models \cite{AAJV-OSR,GHBSW-DSSJS,GHBSWVZ-PDR,GZDHBHSW-RLR,HB-AIR,RBB-DPP,RBB-RSSB,JSW-ERTLR}, the stability condition for c.o.c.\ redundancy is only known in a few (special) cases.
 
\textit{Gardner et al.} \cite{GZDHBHSW-RLR} considered a scenario with i.i.d.\ replicas, exponential job sizes and the FCFS scheduling discipline. They proved that the stability condition is given by $\rho < 1$, where $\rho:=\frac{\lambda \mathbb{E}[X]}{N}$ denotes the load, $\lambda$ the arrival rate of the jobs, $N$ the number of servers and $\mathbb{E}[X]$ the expected job size. Note that the stability condition is independent of the number of replicas. In addition, they obtained an analytical expression for the expected latency.

\textit{Anton et al.} \cite{AAJV-OSR} examined the stability conditions for various scheduling disciplines, both for i.i.d.\ and identical replicas, and exponential job sizes. They showed that for i.i.d.\ replicas the stability conditions for FCFS, PS and Random Order of Service (ROS) are the same. For identical replicas the stability condition for PS is given by $\rho<1/d$, and the maximum sustainable load is thus inversely proportional to the number of replicas. The proofs rely on scaling limits of appropriate lower- and upperbound systems.

For general job size distributions, the stability condition remains an open problem. For i.i.d.\ replicas, the FCFS scheduling discipline and a specific non-exponential distribution, referred to as the scaled Bernoulli distribution, the stability condition is asymptotically given by $\rho N = \lambda \mathbb{E}[X] < K^{d-1}$, where $K$ is the scale of the job size \cite{RBB-RSSB}. Observe that the stability condition is independent of the number of servers, but depends on the scale parameter $K$. For identical replicas there are, to the best of our knowledge, no explicit expressions for the stability condition for any scheduling discipline. A summary of the stability condition results, except the trivial cases $d=1$ and $d=N$, is provided in Table~\ref{tab: stability condition results}.

The present paper focuses on the stability condition for the PS discipline and extends the results from \cite{AAJV-OSR} to general job size distributions with possible dependence among the replicas. This covers the extreme scenarios of perfect dependence (identical replicas) and no dependence at all (i.i.d.\ replicas), as previously considered in the literature, as special cases. The PS discipline is highly relevant as an idealization of Round-Robin scheduling policies that are widely implemented in time-shared computer systems for fairness reasons. To prove the stability condition we consider carefully chosen lower- and upperbound systems. The lowerbound system is similar to the lowerbound system proposed in \cite{AAJV-OSR}. An alternative novel representation of this lowerbound system allows us to prove a necessary stability condition for general job size distributions, via the use of fluid limits. The upperbound system gives an indication for the sufficient stability condition, again via the use of fluid limits. 

The key property of these fluid limits is that all the components of the fluid limit in the lower- and upperbound systems remain equal at all times when starting with the same initial conditions. More importantly, even the fluid limits of the lower- and upperbound systems coincide when starting with the same initial conditions. Hence the stability conditions for the fluid limits of the lower- and upperbound systems coincide, and thus immediately yields the stability condition for the fluid limit of the original system.

The above two properties of the fluid limit reflect that the queue lengths in the original system have the tendency to remain equal when starting from large initial values. This implies that each of the $d$ replicas of a given job will be served at \textit{approximately} the same rate, and hence the total amount of service capacity consumed by an arbitrary job is $d \mathbb{E}[\min\{X_{1},\dots,X_{d}\}]$, yielding the stability condition $\tilde{\rho} := \frac{d \lambda \mathbb{E}[\min\{X_{1},\dots,X_{d}\}]}{N}<1$. Note that for identical replicas the stability condition is insensitive to the job size distribution given its mean, and the maximum sustainable load is inversely proportional to the number of replicas. For i.i.d.\ replicas, $\mathbb{E}[\min\{X_{1},\dots,X_{d}\}]$ decreases and increases in the number of replicas for job size distributions that are New-Better-than-Used (NBU) and New-Worse-than-Used (NWU), respectively, and thus the stability threshold behaves accordingly. 

A key question that remains is the broader issue under which assumptions the (in)stability of the original stochastic process can be deduced from the (in)stability of the fluid limit. While there are strong results in this context \cite{D-FLI,D-PHRFL,DH-PNFMS}, they do not directly cover the combination of general job size distributions and non-head-of-the-line service disciplines like PS. We will further discuss this matter in Section~\ref{sec: stability conditions stochastic system}.

The remainder of the paper is organized as follows. In Section~\ref{sec: model description} we present a detailed description of the redundancy-$d$ model with the PS discipline. The corresponding stability conditions for the fluid limits of both the lower- and upperbound systems, and therefore also of the original, are proved in Section~\ref{sec: stability analysis}. Section~\ref{sec: numerical results} contains numerical experiments that show the accuracy of the lower- and upperbound systems used to prove the stability condition. Section~\ref{sec: conclusion} contains conclusions and some suggestions for further research.

\begin{table}[]
\centering
\tabcolsep=0.11cm
\caption{Summary of the stability condition results. The $X$ indicates a scenario for which we obtain results.}
\label{tab: stability condition results}
\begin{tabular}{|l|c|c|c|c|c|}
\hline
& \multicolumn{2}{c|}{Exponential} & \multicolumn{1}{c|}{Bernoulli} & \multicolumn{2}{c|}{General} \\ \hline
& i.i.d.\ & identical & i.i.d.\ & i.i.d.\ & identical \\ \hline
FCFS & $\rho < 1$ \cite{GZDHBHSW-RLR} &  & $\rho N < K^{d-1}$ \cite{RBB-RSSB} & & \\ \hline
PS & $\rho < 1$ \cite{AAJV-OSR} & $\rho < \frac{1}{d}$ \cite{AAJV-OSR} &  & $X$ & $X$  \\ \hline
ROS & $\rho < 1$ \cite{AAJV-OSR} & $\rho < 1$ \cite{AAJV-OSR} &&& \\ \hline
\end{tabular}
\end{table}

\section{Model description}
\label{sec: model description}
Consider a system with $N$ parallel servers where jobs arrive as a Poisson process of rate $\lambda$. Each of the $N$ servers follows a processor-sharing (PS) discipline. In this paper we focus on the case of homogeneous server speeds, but some of the statements can be extended to heterogeneous server speeds as further discussed in Section~\ref{sec: conclusion}. When a job arrives, the dispatcher immediately assigns replicas to $d \leq N$ servers selected uniformly at random (without replacement). As soon as the first of these $d$ replicas finishes service, the remaining ones are abandoned. 
We allow the replica sizes $X_{1},\dots,X_{d}$ of a job to be governed by some joint distribution $F(x_{1},\dots,x_{d})$, where $X_{i}$, $i=1,\dots,d$, are each distributed as a generic random variable $X$, but not necessarily independent. By Sklar's Theorem, see e.g., \cite{N-IC}, we know that any joint distribution can be written in terms of marginal distribution functions and a copula $C: [0,1]^{d} \rightarrow [0,1]$ which describes the dependency structure between the variables. Note that we do not allow the dependency structure to depend on the state of the system.
Special cases of the dependency structure are i) perfect dependency, so-called identical replicas, where the job size is preserved for all replicas, i.e., $X_{i}=X$, $i=1,\dots,d$, ii) no dependency at all, so-called i.i.d.\ replicas. 

Observe that in this paper the load of the system is defined by $\tilde{\rho} = \frac{d \lambda \mathbb{E}[\min\{X_{1},\dots,X_{d}\}]}{N}$, as opposed to the load $\rho=\frac{\lambda \mathbb{E}[X]}{N}$ defined in \cite{AAJV-OSR,GZDHBHSW-RLR} and used in Table~\ref{tab: stability condition results}.

\section{Stability analysis}
\label{sec: stability analysis}
In this section we analyze the stability condition for redundancy-$d$ systems with PS. In Sections~\ref{sec: lowerbound} and~\ref{sec: upperbound} carefully chosen lower- and upperbound systems are introduced. In Section~\ref{sec: fluid limit} we consider the associated fluid-limit models of these lower- and upperbound systems to prove the corresponding stability conditions. Figure~\ref{fig: overview stability} depicts an overview of this section.

\begin{figure}[H]
  \centering    
    \resizebox{0.75\textwidth}{!}{\begin{tikzpicture}
\draw[<-] node[align=center,left]{Lowerbound\\ system} (0,0) -- node[red,above]{\footnotesize Sec.~$3.1$} (1,0) node[align=center,right]{Stochastic model\\ Redundancy PS};
\draw[->]  (3.9,0) -- node[red,above]{\footnotesize Sec.~$3.2$} (4.9,0) node[align=center,right]{Upperbound\\ system};
\draw[->] (-1,-0.5) -- node[red,right]{\footnotesize Sec.~$3.3$} (-1,-1) node[below]{Fluid model};
\draw[->] (5.9,-0.5) -- node[red,left]{\footnotesize Sec.~$3.3$} (5.9,-1) node[below]{Fluid model};
\draw[->] (-1,-1.5) -- node[red,right]{\footnotesize Lemmas~$2$+$3$} (-1,-2) node[below]{Stability};
\draw[->] (5.9,-1.5) -- node[red,left]{\footnotesize Lemmas~$2$+$3$} (5.9,-2) node[below]{Stability};
\draw (0,-2.25) -- (1,-2.25);
\draw (3.9,-2.25) -- (4.9,-2.25);
\node at (2.55,-2.25) {both coincide};
\end{tikzpicture}}
    \caption{Overview of the various systems.}
    \label{fig: overview stability}
\end{figure}
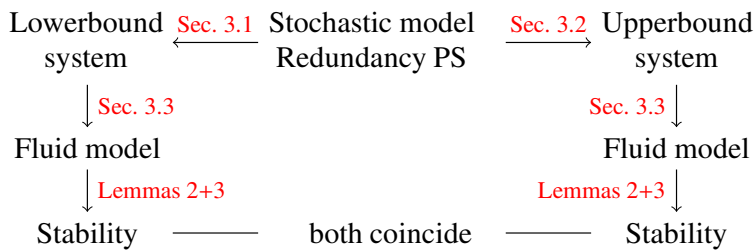

\subsection{Lowerbound system}
\label{sec: lowerbound}

In this section we introduce a lowerbound system which is the same as our original system except for two differences. 
In the lowerbound system replicas always receive a service rate equal to the reciprocal of the minimum queue length, at that time, of all the sampled servers, instead of a service rate equal to the reciprocal of the queue length at that specific server.
The second difference is that in the lowerbound system the sizes of all the replicas are equal to $X_{\text{min}} := \min\{X_{1},\dots,X_{d}\}$, the minimum size of the $d$ replicas in the original system. However, this latter difference actually just follows from the first difference. Namely, if all replicas receive service at equal rate then it immediately follows that the smallest replica will always complete service first, after which the other replicas are abandoned. 

The above-described system provides a lowerbound in the sense that the residual sizes of the replicas are always smaller than or equal to the residual sizes of the replicas in the original system. In particular, the sojourn time of a job in the lowerbound system is always smaller than or equal to that in the original system. The ordering of the residual sizes of the replicas can be established with sample-path arguments for so-called `monotonic' processor-sharing networks, see Lemma~$1$ in \cite{BP-SBPS}.

Next, we give an example to further clarify the operation of the lowerbound system.

\begin{example}
\label{exam: service rate lowerbound}
Assume that the queue lengths at the $d=2$ sampled servers are $(4,5)$. Then, in the original system the replicas receive service at rates $1/4$ and $1/5$, respectively. In the lowerbound system both replicas receive service at rate $1/4= 1/\min\{4,5\} = \max\{1/4,1/5\}$. 
\end{example}

\begin{remark}
For $d=1$ and $d=N$ this lowerbound system is exactly the same as the original system and in the special case of identical replicas this lowerbound system is exactly the same as the lowerbound system introduced in \cite{AAJV-OSR}.
\end{remark}

We now provide an alternative way of viewing the lowerbound system, without any redundant replicas. 

Consider $M={N \choose d}$ (fictitious) job classes, where each job class corresponds to one of the $N \choose d$ possible combinations of $d$ servers. Let $\boldsymbol{s}^{i} \subseteq \{1,\dots,N\}$ denote the set of servers corresponding to the $i$-th job class. Without loss of generality, we suppose that job class $1$ corresponds to the set of servers $\boldsymbol{s}^{1}=\{1,\dots,d\}$, job class $2$ to the set $\boldsymbol{s}^{2}=\{1,\dots,d-1,d+1\}$ and finally job class $M$ to the set $\boldsymbol{s}^{M}=\{N-d+1,\dots,N\}$. All the jobs of a certain class receive service at the same rate, which depends on the number of jobs present of other classes, namely those classes that correspond to sets of servers that have a server in common. Thus, we can also see this as $M$ individual (virtual) queues, one for each job class, that follow a PS discipline with a rate that depends on the number of jobs present at the other virtual queues. Note that only jobs from class $i$ arrive at virtual queue $i$ and these jobs arrive according to a Poisson process with rate $\lambda_{i} = \lambda / M$, see also Figure~\ref{fig: equivalence systems}. 

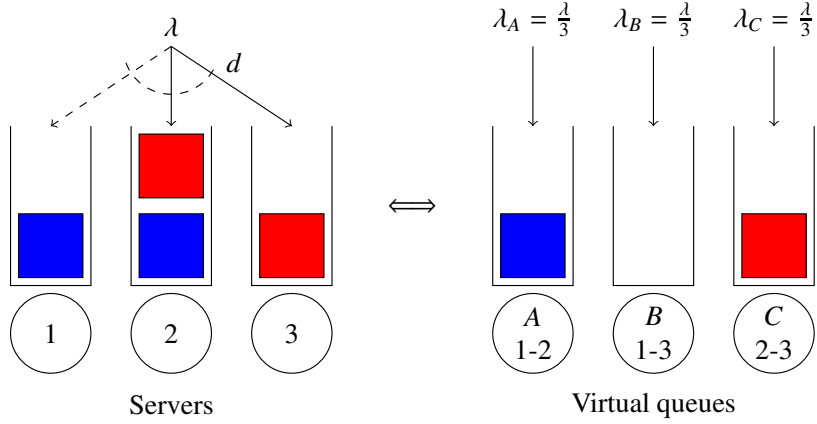
\begin{figure}[htbp]
  \centering
    \resizebox{0.8\textwidth}{!}{%\documentclass[crop,tikz]{standalone}
%\documentclass[tikz]{article}
%\usepackage{tikz}
%\begin{document}

\begin{tikzpicture}[scale=1]
\draw (0,2) -- (0,0) -- (1,0) -- (1,2);
\draw (1.5,2) -- (1.5,0) -- (2.5,0) -- (2.5,2);
\draw (3,2) -- (3,0) -- (4,0) -- (4,2);
\draw (0.5,-0.6) circle [radius=0.5cm] node{$1$};
\draw (2,-0.6) circle [radius=0.5cm] node{$2$};
\draw (3.5,-0.6) circle [radius=0.5cm] node{$3$};
\node at (2,-1.5) {Servers};

\filldraw[fill=blue] (0.1,0.1) rectangle (0.9,0.9);
\filldraw[fill=blue] (1.6,0.1) rectangle (2.4,0.9);
\filldraw[fill=red] (1.6,1.1) rectangle (2.4,1.9);
\filldraw[fill=red] (3.1,0.1) rectangle (3.9,0.9);

\draw[dashed,->] (2,3)--(0.5,2);
\draw[->] (2,3)node[above]{$\lambda$}--(2,2);
\draw[->] (2,3)--(3.5,2);
\draw[dashed] ([shift=(200:0.6cm)]2,3) arc (200:340:0.6cm) node[right]{$d$};

\draw (6,2) -- (6,0) -- (7,0) -- (7,2);
\draw (7.5,2) -- (7.5,0) -- (8.5,0) -- (8.5,2);
\draw (9,2) -- (9,0) -- (10,0) -- (10,2);
\draw (6.5,-0.6) circle [radius=0.5cm] node[align=center]{$A$\\$1$-$2$};
\draw (8,-0.6) circle [radius=0.5cm] node[align=center]{$B$\\$1$-$3$};
\draw (9.5,-0.6) circle [radius=0.5cm] node[align=center]{$C$\\$2$-$3$};
\node at (8,-1.5) {Virtual queues};

\filldraw[fill=blue] (6.1,0.1) rectangle (6.9,0.9);
\filldraw[fill=red] (9.1,0.1) rectangle (9.9,0.9);

\draw[->] (6.5,3)node[above]{$\lambda_{A}=\frac{\lambda}{3}$}--(6.5,2);
\draw[->] (8,3)node[above]{$\lambda_{B}=\frac{\lambda}{3}$}--(8,2);
\draw[->] (9.5,3)node[above]{$\lambda_{C}=\frac{\lambda}{3}$}--(9.5,2);

\node at (5,1) {$\Longleftrightarrow$};
\end{tikzpicture}

%\end{document}
}
    \caption{Visualization of the lowerbound systems.}
    \label{fig: equivalence systems}
\end{figure}

\begin{example}
\label{exam: example lowerbound}
Consider the example visualized in Figure~\ref{fig: equivalence systems} with $N=3$ servers and $d=2$ replicas. In the lowerbound system we have $M = 3$ separate $M/X_{\text{min}}/1/PS$ virtual queues, for clarity denoted by $A$, $B$ and $C$. Moreover $\boldsymbol{s}^{A}=\{1,2\}$, $\boldsymbol{s}^{B}=\{1,3\}$ and $\boldsymbol{s}^{C}=\{2,3\}$. Note that in general the number of servers in the original system and the number of virtual queues may differ. Let $Q^{*}_{j}$ denote the number of jobs at server $j$, $j \in \{1,2,3\}$, and $Q_{i}$ the number of jobs at virtual queue $i$, $i \in \{A,B,C\}$. Observe that at server $2$ on the left hand side two types of jobs can arrive, i.e., jobs which have replicas at servers $1$ and $2$ (this corresponds to job class $A$ on the right hand side) and jobs which have replicas at servers $2$ and $3$ (this corresponds to job class $C$ on the right hand side). Thus, the queue length at server $2$ is equal to the sum of the queue lengths at virtual queues $A$ and $C$, i.e., $Q^{*}_{2}=Q_{A}+Q_{C}$. Therefore, the blue job receives service at rate $1=1/\min\{Q^{*}_{1},Q^{*}_{2}\} = 1/\min\{Q_{A}+Q_{B},Q_{A}+Q_{C}\}$. 

In this example virtual queue $A$ is working at rate $\frac{Q_{A}}{\min\{Q_{A}+Q_{B},Q_{A}+Q_{C}\}}$. Equivalently, virtual queues $B$ and $C$ are working at rates\\ $\frac{Q_{B}}{\min\{Q_{A}+Q_{B},Q_{B}+Q_{C}\}}$ and $\frac{Q_{C}}{\min\{Q_{A}+Q_{C},Q_{B}+Q_{C}\}}$, respectively. 
\end{example}

In general, the service rate, in the lowerbound system, at virtual queue $i$ is 
\begin{align*}
\frac{Q_{i}}{\min\left\{\sum_{i_{1} \in \mathcal{S}_{1}^{i}} Q_{i_{1}},\dots,\sum_{i_{d} \in \mathcal{S}_{d}^{i}}Q_{i_{d}}\right\}},
\end{align*}
where $\mathcal{S}_{j}^{i} = \{k : s^{i}_{j} \in \boldsymbol{s}^{k}\}$ denotes the set of job classes which share server $s^{i}_{j}$ in their corresponding server set and $Q_{i}$ the number of jobs at virtual queue $i$. Thus $\sum_{i_{1} \in \mathcal{S}_{1}^{i}} Q_{i_{1}}$ is the number of jobs at server $s_{1}^{i}$. In Example~\ref{exam: example lowerbound} we have $\mathcal{S}_{1}^{A}=\mathcal{S}_{1}^{B}=\{A,B\}$, $\mathcal{S}_{2}^{A}=\mathcal{S}_{1}^{C}=\{A,C\}$ and $\mathcal{S}_{2}^{B}=\mathcal{S}_{2}^{C}=\{B,C\}$. 

Again, note that this alternative lowerbound system, explained in Example~\ref{exam: example lowerbound}, behaves exactly the same as the lowerbound system described at the beginning of this subsection, in the sense that a job at virtual queue $i$ receives exactly the same service rate as the $d$ replicas at the servers in the original system that belong to $\boldsymbol{s}^{i}$. However, in the alternative lowerbound system there is no replication of jobs, i.e., the dependency is captured in the service rate, which makes the alternative lowerbound system easier to analyze. 

\subsection{Upperbound system}
\label{sec: upperbound}
In this section we introduce an upperbound system which is the same as our original system except for two differences. In the upperbound system replicas always receive a service rate equal to the reciprocal of the maximum queue length, at that time, of all the sampled servers, instead of a service rate equal to the reciprocal of the queue length at that specific server. The second difference is that in the upperbound system the sizes of all the $d$ replicas are equal to $X_{\text{min}}=\min\{X_{1},\dots,X_{d}\}$, the minimum size of the $d$ replicas in the original system. Again, similar to the lowerbound system, note that this latter difference is a direct implication of the first difference since all the $d$ replicas receive service at equal rate. Therefore the smallest replica will always complete service first. 

The above-described system provides an upperbound in the sense that the residual sizes of the replicas are always larger than or equal to the residual sizes of the replicas in the original system. Once again, this ordering can be proved with sample-path arguments for so called monotonic processor-sharing networks. 

For the upperbound system we give an example similar to Example~\ref{exam: service rate lowerbound}.
\begin{example}
Assume that the queue lengths at the $d=2$ sampled servers are $(4,5)$. Then, in the original system the replicas receive service at rates $1/4$ and $1/5$, respectively. In the upperbound system both replicas receive service at a rate $1/5= 1/\max\{4,5\} = \min\{1/4,1/5\}$. 
\end{example}

\begin{remark}
For $d=1$ and $d=N$ this upperbound system is exactly the same as the original system.
\end{remark}

Just like for the lowerbound system, we now provide an alternative way of viewing the upperbound system. This is accomplished by taking the same set of $M={N \choose d}$ virtual queues. The only difference in Example~\ref{exam: example lowerbound} is that now the blue job receives service at rate $1=1/\max\{Q^{*}_{1},Q^{*}_{2}\} = 1/\max\{Q_{A}+Q_{B},Q_{A}+Q_{C}\}$.

In general, the service rate, in the upperbound system, at virtual queue $i$ is 
\begin{align*}
\frac{Q_{i}}{\max\left\{\sum_{i_{1} \in \mathcal{S}_{1}^{i}} Q_{i_{1}},\dots,\sum_{i_{d} \in \mathcal{S}_{d}^{i}}Q_{i_{d}}\right\}}.
\end{align*}

Because of the similarities of the lower- and upperbound systems, the fluid limits are very similar as well. Therefore we discuss both fluid limits simultaneously in the next section. 

\subsection{Fluid-limit models}
\label{sec: fluid limit}
Let $Q_{i}(t)$ denote the number of jobs at virtual queue $i$ at time $t$ with $\boldsymbol{Q}(t)=(Q_{1}(t),\dots,Q_{M}(t))$. 
The time until virtual queue $i$ first becomes empty is denoted by $\tau_{i} = \inf \{t | Q_{i}(t) = 0\}$. 
The service rate per job at virtual queue $i$ at time $s$ is given by 
\begin{align}
\label{eq: service rate lowerbound}
\varphi_{i}\left(\boldsymbol{Q}(s)\right) = \frac{1}{\min\left\{\sum_{i_{1} \in \mathcal{S}_{1}^{i}} Q_{i_{1}}(s),\dots,\sum_{i_{d} \in \mathcal{S}_{d}^{i}}Q_{i_{d}}(s)\right\}},
\end{align}
for the lowerbound system and 
\begin{align}
\label{eq: service rate upperbound}
\varphi_{i}\left(\boldsymbol{Q}(s)\right) = \frac{1}{\max\left\{\sum_{i_{1} \in \mathcal{S}_{1}^{i}} Q_{i_{1}}(s),\dots,\sum_{i_{d} \in \mathcal{S}_{d}^{i}} Q_{i_{d}}(s)\right\}},
\end{align}
for the upperbound system whenever $Q_{i}(s)>0$ and $\varphi_{i}(\boldsymbol{Q}(s))=0$ for $Q_{i}(s)=0$. 
The attained service process, i.e., the cumulative amount of processing time per job allocated to virtual queue $i$ during the interval $[s,t]$, is
\begin{align*}
\eta_{i}(s,t) = \int_{u=s}^{t} \varphi_{i}(\boldsymbol{Q}(u)) \text{d}u.
\end{align*}
For $0 \leq t \leq \tau_{i}$, $\eta_{i}(0,t)$ is continuous and strictly increasing from an initial value of 0, and for $t > \tau_{i}$ it is (Lipschitz) continuous nondecreasing in the second argument. 

Let $U_{ik}$ be the arrival time of the $k$th job at virtual queue $i$. Define $A_{i}(t) = \max\{k: U_{ik} < t\}$ as the number of jobs that arrive at virtual queue $i$ during the interval $(0,t]$. The virtual queue length process satisfies
\begin{align}
\label{eq: virtual queue length process}
Q_{i}(t) = \sum_{l=1}^{Q_{i}(0)} \mathbbm{1}_{\{v'_{il} > \eta_{i}(t)\}} + \sum_{k=1}^{A_{i}(t)} \mathbbm{1}_{\{v_{ik} > \eta_{i}(U_{ik},t)\}},
\end{align}
where $v'_{il}$ denotes the residual job size of the $l$th initial job at virtual queue $i$ and $v_{ik}$ the job size of the $k$th arriving job at virtual queue $i$.

We consider the behavior of the system on a fluid scale and define the scaled processes
\begin{align*}
&\overline{Q}_{i}^{n}(t) = Q_{i}(nt)/n, \\
&\overline{\eta}_{i}^{n}(s,t) = \int_{u=s}^{t} \varphi_{i}(\boldsymbol{\overline{Q}}^{n}(u))\text{d}u = \int_{u=s}^{t} n \varphi_{i}(\boldsymbol{Q}(n u))\text{d}u = \eta_{i}(ns,nt), \\
& \overline{\tau}_{i}^{n} = \inf \{t | \overline{Q}_{i}^{n}(t)=0 \} = \tau_{i}/n.
\end{align*}
The scaled attained service time is continuous and strictly increasing for $0 \leq t < \overline{\tau}_{i}^{n}$, and continuous non-decreasing for $t \geq \overline{\tau}_{i}^{n}$ in the second argument.

\begin{definition}
A non-negative continuous function $\overline{Q}_{i}(\cdot)$ is a fluid-model solution if it satisfies the functional equation
\begin{align}
\label{eq: fluid limit}
\overline{Q}_{i}(t) = \overline{Q}_{i}(0) (1-G(\overline{\eta}_{i}(0,t))) + \frac{\lambda}{M} \int_{s=0}^{t} (1-F_{X_{\text{min}}}(\overline{\eta}_{i}(s,t))) \text{d}s,
\end{align}
where $G(\cdot)$ is the service time distribution of initial jobs, $F_{X_{\text{min}}}(\cdot)$ the service time distribution of arriving jobs and
\begin{align*}
\overline{\eta}_{i}(s,t) = \int_{u=s}^{t} \varphi_{i}(\overline{\boldsymbol{Q}}(u)) \text{d}u.
\end{align*}
\end{definition}

\begin{theorem}
\label{thm: fluid limit}
The limit point of any convergent subsequence of $(\overline{Q}_{i}^{n}(t); t \geq 0)$ is almost surely a solution of the fluid-model Equation~\eqref{eq: fluid limit}.
\end{theorem}
\begin{proof}
The proof is presented in~\ref{app sec: proof fluid limits}.
\end{proof}

Next, we prove a monotonicity result for the fluid limit given by Equation~\eqref{eq: fluid limit}. 

\begin{definition}
We say that the service rate per job is monotonic if 
\begin{align*}
\frac{\varphi_{i}(\overline{\boldsymbol{Q}}^{a}(t))}{\overline{Q}_{i}^{a}(t)} \geq \frac{\varphi_{i}(\overline{\boldsymbol{Q}}^{b}(t))}{\overline{Q}_{i}^{b}(t)}, ~~~ \forall \overline{\boldsymbol{Q}}^{a}(t) \leq \overline{\boldsymbol{Q}}^{b}(t) : \overline{Q}_{i}^{a}(t) > 0, 
\end{align*}
see also \cite{BP-SBPS}.
\end{definition}

\begin{lemma}
\label{lem: monotonicity initial jobs}
Let $\varphi_{i}(\cdot)$ be monotonic for all $i=1\dots,M$ and consider the fluid limits with initial conditions $\overline{\boldsymbol{Q}}^{a}(0)$ and $\overline{\boldsymbol{Q}}^{b}(0)$, where $\overline{\boldsymbol{Q}}^{a}(0) \leq \overline{\boldsymbol{Q}}^{b}(0)$, then 
$\varphi_{i}(\overline{\boldsymbol{Q}}^{a}(t)) \geq \varphi_{i}(\overline{\boldsymbol{Q}}^{b}(t))$
for all $i=1\dots,M$ and $t \geq 0$ for which $\overline{Q}_{i}^{a}(t) > 0$.
\end{lemma}
\begin{proof}
Assume that there exists an $i \in \{1,\dots,M\}$ such that $t_{0}$ is the first time for which $\varphi_{i}(\overline{\boldsymbol{Q}}^{a}(t_{0})) < \varphi_{i}(\overline{\boldsymbol{Q}}^{b}(t_{0}))$ with $\overline{Q}_{i}^{a}(t) > 0$. Then, there must exist an $i^{*} \in \{1,\dots,M\}$ such that 
$\overline{Q}_{i^{*}}^{a}(t) > \overline{Q}_{i^{*}}^{b}(t)$
for some $0 \leq t < t_{0}$.
However, $\varphi_{i}(\overline{\boldsymbol{Q}}^{a}(t)) \geq \varphi_{i}(\overline{\boldsymbol{Q}}^{b}(t))$, for all $i=1,\dots,M$ and $0 \leq t < t_{0}$, from which it follows that 
$\overline{Q}_{i}^{a}(t) \leq \overline{Q}_{i}^{b}(t)$
for all $i=1,\dots,M$ and $0 \leq t < t_{0}$.
\end{proof}

The key property of these fluid limits is that all the components of the fluid limit corresponding to all the virtual queues remain equal at all times when starting with the same initial conditions. More importantly, even the fluid limits of the lower- and upperbound systems coincide when starting with the same initial conditions for all the virtual queues, since Equations~\eqref{eq: service rate lowerbound} and ~\eqref{eq: service rate upperbound} coincide. 

\begin{property}
\label{prop: fluid limit}
If $\overline{Q}_{i}(0)= q(0)$ for all $i=1,\dots,M$, then $\overline{Q}_{i}(t)= q(t)$ for all $i=1,\dots,M$ and $t \geq 0$, where
\begin{align*}
q(t) = q(0) (1-G(\overline{\eta}(0,t))) + \frac{\lambda}{M} \int_{s=0}^{t} (1-F_{X_{\text{min}}}(\overline{\eta}(s,t))) \text{d}s,
\end{align*}
with
\begin{align*}
\overline{\eta}(s,t) = \frac{1}{{N-1 \choose d-1}} \int_{u=s}^{t} \varphi(q(u))\text{d}u,
\end{align*}
and
\begin{align*}
\varphi(q(u)) = \frac{1}{q(u)},
\end{align*}
for $q(u) > 0$ and $\varphi(0)=0$. The expression for the attained service process follows from $\varphi_{i}(\boldsymbol{Q}(s)) = \frac{1}{{N-1 \choose d-1}q(s)}$ for all $i=1,\dots,M$.
Note that this is the fluid limit of an $M/X_{\text{min}}/1/PS$ queue with arrival rate $\lambda / M$ and server speed $1/{N-1 \choose d-1}$.
\end{property}

The proof of Property~\ref{prop: fluid limit} follows by contradiction. Assume that there exists an $i \in \{1,\dots,M\}$ such that $t_{0}$ is the first time for which $\overline{Q}_{i}(t_{0}) \neq q(t_{0})$. Then there must exist an $i^{*} \in \{1,\dots,M\}$ such that $\overline{\eta}_{i^{*}}(s,t_{1}) \neq \overline{\eta}(s,t_{1})$ with $0 \leq s \leq t_{1} < t_{0}$. However, $\overline{Q}_{i}(t)= q(t)$ for all $i=1,\dots,M$ and $t \leq t_{1}$ from which it follows $\overline{\eta}_{i}(s,t)= \overline{\eta}(s,t)$ for all $i=1,\dots,M$ and $0 \leq s \leq t \leq t_{1}$.

\subsection{Stability conditions for fluid-limit models}
\label{sec: stability conditions fluid limit}
We now turn to (in)stability conditions for the fluid-limit models, where a fluid-limit model is called stable of it reaches zero in finite time for any given initial state and unstable if it tends to infinity over time. 

\begin{lemma}
\label{lem: necessary stability fluid}
A necessary stability condition for the fluid-limit model of both the lower- and upperbound system as well as the original system is given by $\tilde{\rho}\leq 1$.
\end{lemma}
\begin{proof}
For the minimum of the fluid-limit components we have that
\begin{align*}
\min_{i \in \{1,\dots,M\}} \overline{Q}_{i}(t) \geq
\min_{i \in \{1,\dots,M\}} \overline{Q}_{i}^{0}(t) =
\frac{\lambda}{M} \Big( t - \int_{s=0}^{t} F_{X_{\text{min}}}(\overline{\eta}^{*}(s,t)) \text{d}s \Big),
\end{align*} 
where $\overline{Q}_{i}^{0}(t)$ denotes the fluid limit and $\overline{\eta}^{*}$ the attained service at virtual queue $i$ when starting with initial condition $0$ at all virtual queues. The inequality follows from Lemma~\ref{lem: monotonicity initial jobs} and the equality follows by Property~\ref{prop: fluid limit}. Note that the above expression holds both for the lower- and upperbound system, with the service rate per job given by Equations~\eqref{eq: service rate lowerbound} and~\eqref{eq: service rate upperbound}, respectively.

Thus, the minimum of the fluid-limit components is lower bounded by the fluid limit of an $M/X_{\text{min}}/1/PS$ queue with arrival rate $\lambda / M$ and server speed $1/{N-1 \choose d-1}$. This fluid limit goes to infinity for $\tilde{\rho}=\frac{d \lambda \mathbb{E}[\min\{X_{1},\dots,X_{d}\}]}{N} > 1$, since ${N-1 \choose d-1}/M = d/N$. 
\end{proof}

\begin{lemma}
\label{lem: sufficient stability fluid}
A sufficient stability condition for the fluid-limit model of both the lower- and upperbound system as well as the original system is given by $\tilde{\rho} < 1$.
\end{lemma}
\begin{proof}
For the maximum of the fluid-limit components we have that
\begin{align*}
\max_{i \in \{1,\dots,M\}} \overline{Q}_{i}(t) &\leq \max_{i \in \{1,\dots,M\}} \overline{Q}_{i}^{q}(t) = \left( \max_{j \in \{1,\dots,M\}} \overline{Q}_{j}(0) \right) \cdot (1-G(\overline{\eta}^{*}(0,t))) + \frac{\lambda}{M} \Big( t - \int_{s=0}^{t} F_{X_{\text{min}}}(\overline{\eta}^{*}(s,t)) \text{d}s \Big),
\end{align*} 
where $\overline{Q}_{i}^{q}(t)$ denotes the fluid limit and $\overline{\eta}^{*}$ the attained service at virtual queue $i$ when starting with initial condition \\$q=\max_{j \in \{1,\dots,M\}} \overline{Q}_{j}(0)$ at all virtual queues. The inequality follows from Lemma~\ref{lem: monotonicity initial jobs} and the equality follows by Property~\ref{prop: fluid limit}. Once again, note that the above expression holds both for the lower- and upperbound system, with the service rate per job given by Equations~\eqref{eq: service rate lowerbound} and~\eqref{eq: service rate upperbound}, respectively.

Thus, the maximum of the fluid-limit components is upper bounded by the fluid limit of an $M/X_{\text{min}}/1/PS$ queue with arrival rate $\lambda / M$ and server speed $1/{N-1 \choose d-1}$. This fluid limit goes to zero in finite time for $\tilde{\rho}=\frac{d \lambda \mathbb{E}[\min\{X_{1},\dots,X_{d}\}]}{N} < 1$. 
\end{proof}

Now, we discuss the stability condition for the fluid-limit model of the original system in more detail.  

\begin{definition}
\label{def: NBU NWU distribution}
Consider a non-negative random variable $X$ with support denoted by $R_{X}$ and cumulative distribution function (cdf) $F_{X}(x)$. Let $\bar{F}_{X}(x) = 1-F_{X}(x)$ denote the complementary cumulative distribution function (ccdf). Then, $X$ is New-Better-than-Used (NBU) (New-Worse-than-Used (NWU)) if for all $t_{1}, t_{2} \in R_{X}$,
\begin{align}
\label{eq: NBU NWU definition}
\bar{F}_{X}(t_{1}+t_{2}) \leq (\geq) \bar{F}_{X}(t_{1}) \bar{F}_{X}(t_{2}).
\end{align}
\end{definition}
Note that 
\begin{align*}
&\mathbb{E}[\min\{X_{1},\dots,X_{d}\}] = \int_{[0,1]^{d}} \min\{F_{X}^{-1}(u_{1}),\dots,F_{X}^{-1}(u_{d})\} \text{d}C(u_{1},\dots,u_{d}),
\end{align*}
where $C$ denotes the copula model describing the dependency structure, see also~\cite{N-IC}. 
In the special case of i.i.d.\ replica sizes and a job size distribution that is NBU (NWU) we have
\begin{align*}
\mathbb{E}[\min\{X_{1},\dots,X_{d}\}] &= \int_{t=0}^{\infty} (1-F_{X}(t))^{d}\text{d}t \geq (\leq) \int_{t=0}^{\infty} (1-F_{X}(d t))\text{d}t = \frac{1}{d} \mathbb{E}[X].
\end{align*}
This implies that for i.i.d.\ replica sizes the stability threshold, compared to the stability threshold for the exponential distribution, is larger for NWU distributions and smaller for NBU distributions. 
Moreover, note that $d \mathbb{E}[\min\{X_{1},\dots,X_{d}\}]$ is increasing and decreasing in $d$ for NBU and NWU distributions and therefore the stability threshold is maximized for $d=1$ (no replication) and $d=N$ (full replication), respectively. The latter property ties in with the result proved in \cite{KR-RAGC} for a slightly different setup that for i.i.d.\ NWU job sizes full replication stochastically maximizes the departure process, and thus achieves maximum stability. 

\subsection{Stability conditions for the original stochastic system}
\label{sec: stability conditions stochastic system}

We now discuss under what assumptions the (in)stability of the original
stochastic process can be deduced from the (in)stability of the fluid limit.

In~\cite{D-FLI} it is proved that if a fluid limit is unstable,
then the original stochastic process is unstable, in the sense that the
number of jobs tends to infinity, under the assumption of service
disciplines where at most $k < \infty$ jobs are (partially) in service
at a time.
This latter assumption is not satisfied for PS.
In~\cite{M-TQNFL} a similar result is proved when the fluid limit tends
to infinity at a linear rate for any initial condition.
While the latter property seems highly plausible in view of the proof
of Lemma~2, this is essentially an issue involving a stand-alone
$M/X/1/PS$ queue, and we will not pursue a rigorous proof here,
which leaves us with the next conjecture.

\begin{conjecture}
\label{conj: necessary stability}
A necessary stability condition for redundancy-$d$ systems with PS
and general job size distributions is given by $\tilde{\rho} \leq 1$.
\end{conjecture}

In~\cite{D-PHRFL} it is proved that the stability of the fluid limit implies
the stability of the stochastic process for general job size distributions.
However, this result assumes \textit{head-of-the-line} service disciplines
and does not cover PS.
For exponential job sizes, the \textit{head-of-the-line} service
disciplines and PS are equivalent, and for the special cases
of identical and i.i.d.\ replicas it follows that $X_{\text{min}}$ is
exponentially distributed with mean $\mathbb{E}[X]$ and $\mathbb{E}[X]/d$,
respectively.
In this special case, we can therefore use the results in~\cite{D-PHRFL}
to obtain the (sufficient) stability condition for the stochastic
upperbound system and recover the results in~\cite{AAJV-OSR}.

Extending the result in~\cite{D-PHRFL} to non-head-of-the-line service
disciplines like PS is beyond the scope of the present paper,
and a highly challenging problem in its own right,
see also \cite{DH-PNFMS,WM-SCBS} for a more extensive discussion.
This leads to the next conjecture.

\begin{conjecture}
\label{conj: sufficient stability}
A sufficient stability condition for redundancy-$d$ systems with PS
and general job size distributions is given by $\tilde{\rho} < 1$.
\end{conjecture}

We now turn attention to an alternative upperbound system that is
analytically tractable and provides a rigorous sufficient stability
condition for the original stochastic system.
Specifically, consider a system where all the $d$~replicas are fully
served, as opposed to at least one server (out of the $d$~sampled servers)
fully completing the replica, see also \cite[Section 6.2.2]{AAJV-OSR}. 
Each individual server in this alternative system can be viewed
as an $M/X/1/PS$ queue with arrival rate $\frac{d \lambda}{N}$
and expected job size $\mathbb{E}[X]$. 
 
\begin{remark}
\label{rem: tightness upperbound independent}
For $d=1$ and $d=N$ this upperbound system is exactly the same as the original system. 
\end{remark} 
 
\begin{theorem}
\label{thm: sufficient stability upperbound}
A sufficient stability condition for redundancy-$d$ systems with the processor-sharing discipline and general job size distributions is given by $\rho < \frac{1}{d}$, with $\rho = \frac{\lambda \mathbb{E}[X]}{N}$.
\end{theorem}

\begin{proof}
Follows by comparison of the original system with the $M/X/1/PS$ queue.
\end{proof}

Note that for the special case of identical replicas this was already
proved in~\cite[Proposition~18]{AAJV-OSR}.

In general the sufficient stability condition
in Theorem~\ref{thm: sufficient stability upperbound} does not coincide
with the necessary stability condition as stated
in Conjecture~\ref{conj: necessary stability}.
However, for identical replicas it does, since in that case
$d \rho = \tilde{\rho}$, resulting in the stability condition
$\tilde{\rho} = \frac{d \mathbb{E}[X]}{N}<1$.
Here, the PS discipline yields the worst possible stability condition
among all work-conserving service disciplines, since the stability
condition corresponds to that in a system where all replicas are fully
served as is the case in the above alternative upperbound system.

\begin{remark}
\label{rem: comparing FCFS with PS scaled bernoulli}
For independent exponentially distributed job sizes the stability
conditions for FCFS and PS are the same, whereas this is not the case for scaled Bernoulli distributed job sizes.
Let the job have sizes $0$ and $K$ with probability $p=1-1/K$ and $1-p$, respectively. Then, the asymptotic stability condition for FCFS is given by $\lambda / K^{d-1} < 1$, see~\cite{RBB-RSSB}.
The stability condition for PS is $(\lambda d) / (N K^{d-1}) < 1$, which is
larger than the stability threshold for FCFS.
\end{remark}

\section{Numerical results}
\label{sec: numerical results}
In Section~\ref{subsec: accuracy bounds} the accuracy of the bounds for both identical and i.i.d.\ replicas is studied. In Section~\ref{subsec: near insensitivity identical replicas} we show near-insensitivity of the expected latency for identical replicas. In Section~\ref{subsec: load} we examine the delay performance as function of the load. 

\begin{table}[htbp]
\centering
\tabcolsep=0.11cm
\caption{The various job size distributions with $\mathbb{E}[X]=2$.}
\label{tab: distributions}
\begin{tabular}{|l|c|l|}
\hline
Distribution & Variance & Description \\ \hline
Deterministic & $0$ & point mass at $2$ \\ \hline
\multirow{2}{*}{Erlang2} & \multirow{2}{*}{$2$} & sum of two exponential random \\
& & variables with mean $1$ \\ \hline
Exponential & $4$ & \\ \hline
\multirow{2}{*}{Bimodal-1} & \multirow{2}{*}{$9$} & takes values $1$ and $11$ with probability  \\ 
& & $0.9$ and $0.1$, respectively \\ \hline
\multirow{2}{*}{Weibull-1} & \multirow{2}{*}{$20$} & shape parameter $=\frac{1}{2}$ and \\ 
& & scale parameter $=1$ (heavy-tailed) \\ \hline
\multirow{2}{*}{Weibull-2 }& \multirow{2}{*}{$76$} & shape parameter $=\frac{1}{3}$ and  \\ 
& & scale parameter $=\frac{1}{3}$ (heavy-tailed) \\ \hline
\multirow{2}{*}{Bimodal-2} & \multirow{2}{*}{$90$} & takes values $1$ and $101$ with probability \\
 & & $0.99$ and $0.01$, respectively \\ \hline
\end{tabular}
\end{table}

\subsection{Accuracy bounds}
\label{subsec: accuracy bounds}
In Figure~\ref{fig: bounds identical replicas} the expected latency for the lowerbound system, upperbound system and original system with $N=4$ (homogeneous) servers and $d=2$ identical replicas is depicted for various job size distributions. Only the alternative upperbound system from Section~\ref{sec: stability conditions stochastic system} is analytically tractable and the expected latency in the other systems is obtained via simulation. 
For identical replicas it can be seen that the lowerbound is quite accurate. Moreover, note that the expected latency in the upperbound system is only dependent on the mean of the job size distribution and not its higher moments. 
At first instance, it seems that the expected latency is insensitive to the job size distribution, but in Section~\ref{subsec: near insensitivity identical replicas} we show that the expected latency in fact slightly differs for the various job size distributions, which is called near-insensitivity. 

\begin{figure}[htbp]
  \centering
    \resizebox{0.49\textwidth}{!}{\newcommand{\dataFigure}{TikZFigures/Redundancy_PS_identical_N4_d2.csv}
\newcommand{\dataFigureLowerbound}{TikZFigures/Redundancy_PS_identical_N4_d2_lowerbound.csv}
\newcommand{\dataFigureUpperbound}{TikZFigures/Redundancy_PS_identical_N4_d2_upperbound.csv}

\begin{tikzpicture}
		\begin{axis}[
			xlabel=$\lambda$,
			ylabel=Expected latency,
			ymin=0,
			ymax=8,
			no markers,
			legend pos= outer north east,
			legend cell align=left]
		\addlegendimage{empty legend}
		\addplot+ table [x=lambda, y=Deterministic, col sep=comma]{\dataFigure};
		\addplot+ table [x=lambda, y=Erlang2, col sep=comma]{\dataFigure};
		\addplot+ table [x=lambda, y=Exponential, col sep=comma]{\dataFigure};
		\addplot+ table [x=lambda, y=Bimodal-1, col sep=comma]{\dataFigure};
		\addplot+ table [x=lambda, y=Weibull-1, col sep=comma]{\dataFigure};
		\addplot+ table [x=lambda, y=Weibull-2, col sep=comma]{\dataFigure};
		\addplot+ table [x=lambda, y=Bimodal-2, col sep=comma]{\dataFigure};
		
		\pgfplotsset{cycle list shift=-7}
		\addplot+[dashed] table [x=lambda, y=Deterministic, col sep=comma]{\dataFigureLowerbound};
		\addplot+[dashed] table [x=lambda, y=Erlang2, col sep=comma]{\dataFigureLowerbound};
		\addplot+[dashed] table [x=lambda, y=Exponential, col sep=comma]{\dataFigureLowerbound};
		\addplot+[dashed] table [x=lambda, y=Bimodal-1, col sep=comma]{\dataFigureLowerbound};
		\addplot+[dashed] table [x=lambda, y=Weibull-1, col sep=comma]{\dataFigureLowerbound};
		\addplot+[dashed] table [x=lambda, y=Weibull-2, col sep=comma]{\dataFigureLowerbound};
		\addplot+[dashed] table [x=lambda, y=Bimodal-2, col sep=comma]{\dataFigureLowerbound};
		
		\addplot+[dashed, black] table [x=lambda, y=all, col sep=comma]{\dataFigureUpperbound};
		\addlegendentry{Distributions:}
		\addlegendentry{Deterministic}
		\addlegendentry{Erlang2}
		\addlegendentry{Exponential}
		\addlegendentry{Bimodal-$1$}
		\addlegendentry{Weibull-$1$}
		\addlegendentry{Weibull-$2$}
		\addlegendentry{Bimodal-$2$}
		\end{axis}
\end{tikzpicture}}
    \resizebox{0.49\textwidth}{!}{\newcommand{\dataFigure}{TikZFigures/Redundancy_PS_identical_N4_d3.csv}
\newcommand{\dataFigureLowerbound}{TikZFigures/Redundancy_PS_identical_N4_d3_lowerbound.csv}
\newcommand{\dataFigureUpperbound}{TikZFigures/Redundancy_PS_identical_N4_d3_upperbound.csv}

\begin{tikzpicture}
		\begin{axis}[
			xlabel=$\lambda$,
			ylabel=Expected latency,
			ymin=0,
			ymax=8,
			no markers,
			legend pos= outer north east,
			legend cell align=left]
		\addlegendimage{empty legend}
		\addplot+ table [x=lambda, y=Deterministic, col sep=comma]{\dataFigure};
		\addplot+ table [x=lambda, y=Erlang2, col sep=comma]{\dataFigure};
		\addplot+ table [x=lambda, y=Exponential, col sep=comma]{\dataFigure};
		\addplot+ table [x=lambda, y=Bimodal-1, col sep=comma]{\dataFigure};
		\addplot+ table [x=lambda, y=Weibull-1, col sep=comma]{\dataFigure};
		\addplot+ table [x=lambda, y=Weibull-2, col sep=comma]{\dataFigure};
		\addplot+ table [x=lambda, y=Bimodal-2, col sep=comma]{\dataFigure};
		
		\pgfplotsset{cycle list shift=-7}
		\addplot+[dashed] table [x=lambda, y=Deterministic, col sep=comma]{\dataFigureLowerbound};
		\addplot+[dashed] table [x=lambda, y=Erlang2, col sep=comma]{\dataFigureLowerbound};
		\addplot+[dashed] table [x=lambda, y=Exponential, col sep=comma]{\dataFigureLowerbound};
		\addplot+[dashed] table [x=lambda, y=Bimodal-1, col sep=comma]{\dataFigureLowerbound};
		\addplot+[dashed] table [x=lambda, y=Weibull-1, col sep=comma]{\dataFigureLowerbound};
		\addplot+[dashed] table [x=lambda, y=Weibull-2, col sep=comma]{\dataFigureLowerbound};
		\addplot+[dashed] table [x=lambda, y=Bimodal-2, col sep=comma]{\dataFigureLowerbound};
		
		\addplot+[dashed, black] table [x=lambda, y=all, col sep=comma]{\dataFigureUpperbound};
		\addlegendentry{Distributions:}
		\addlegendentry{Deterministic}
		\addlegendentry{Erlang2}
		\addlegendentry{Exponential}
		\addlegendentry{Bimodal-$1$}
		\addlegendentry{Weibull-$1$}
		\addlegendentry{Weibull-$2$}
		\addlegendentry{Bimodal-$2$}
		\end{axis}
\end{tikzpicture}}
    \caption{The expected latency in the original system and in the lower- and upperbound systems (dashed lines) for PS with identical replicas, $N=4$, $\mathbb{E}[X]=2$, various job size distributions (see Table~\ref{tab: distributions}) and $d=2$ (left) and $d=3$ (right).
    \label{fig: bounds identical replicas}}
\end{figure}
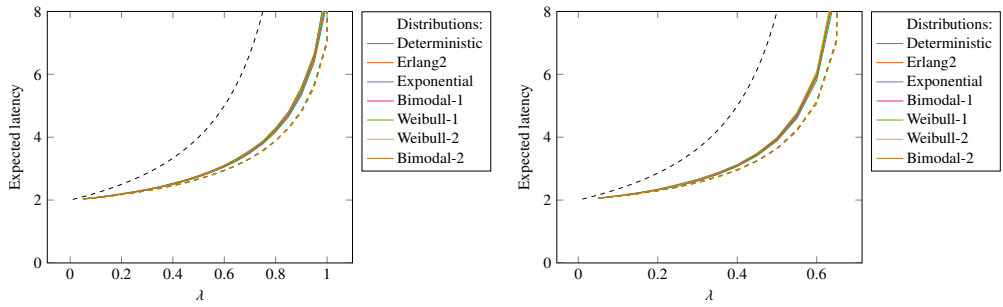

In Figure~\ref{fig: bounds independent replicas} the expected latency in the lowerbound system, upperbound system and original system with $N=4$ (homogeneous) servers and $d=2$ i.id.\ replicas is depicted for various job size distributions.
Note that the lowerbound for deterministic job sizes is tight, see also Remark~\ref{rem: tightness upperbound independent}.

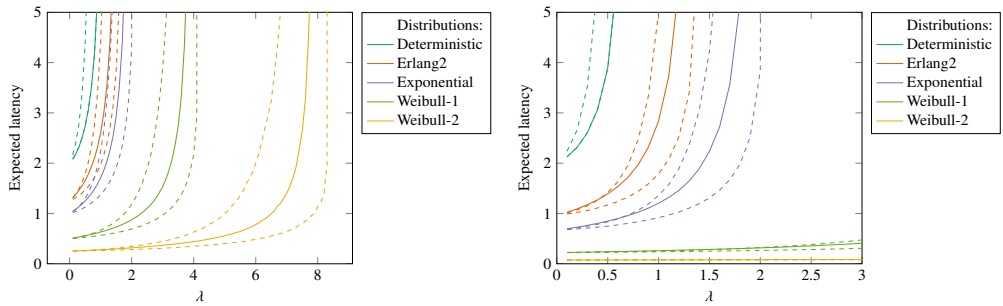
\begin{figure}[htbp]
  \centering
    \resizebox{0.49\textwidth}{!}{\newcommand{\dataFigure}{TikZFigures/Redundancy_PS_independent_N4_d2.csv}
\newcommand{\dataFigureLowerbound}{TikZFigures/Redundancy_PS_independent_N4_d2_lowerbound.csv}
\newcommand{\dataFigureUpperbound}{TikZFigures/Redundancy_PS_independent_N4_d2_upperbound.csv}

\begin{tikzpicture}
		\begin{axis}[
			xlabel=$\lambda$,
			ylabel=Expected latency,
			ymin=0,
			ymax=5,
			no markers,
			legend pos= outer north east,
			legend cell align=left]
		\addlegendimage{empty legend}
		\addplot+ table [x=lambda, y=Deterministic, col sep=comma]{\dataFigure};
		\addplot+ table [x=lambda, y=Erlang2, col sep=comma]{\dataFigure};
		\addplot+ table [x=lambda, y=Exponential, col sep=comma]{\dataFigure};
		\pgfplotsset{cycle list shift=1}
		\addplot+ table [x=lambda, y=Weibull-1, col sep=comma]{\dataFigure};
		\addplot+ table [x=lambda, y=Weibull-2, col sep=comma]{\dataFigure};
		
		 \pgfplotsset{cycle list shift=-5}
		\addplot+[dashed] table [x=lambda, y=Deterministic, col sep=comma]{\dataFigureLowerbound};
		\addplot+[dashed] table [x=lambda, y=Erlang2, col sep=comma]{\dataFigureLowerbound};
		\addplot+[dashed] table [x=lambda, y=Exponential, col sep=comma]{\dataFigureLowerbound};
		\pgfplotsset{cycle list shift=-4}
		\addplot+[dashed] table [x=lambda, y=Weibull-1, col sep=comma]{\dataFigureLowerbound};
		\addplot+[dashed] table [x=lambda, y=Weibull-2, col sep=comma]{\dataFigureLowerbound};
		
		\pgfplotsset{cycle list shift=-10}
		\addplot+[dashed] table [x=lambda, y=Deterministic, col sep=comma]{\dataFigureUpperbound};
		\addplot+[dashed] table [x=lambda, y=Erlang2, col sep=comma]{\dataFigureUpperbound};
		\addplot+[dashed] table [x=lambda, y=Exponential, col sep=comma]{\dataFigureUpperbound};
		\pgfplotsset{cycle list shift=-9}
		\addplot+[dashed] table [x=lambda, y=Weibull-1, col sep=comma]{\dataFigureUpperbound};
		\addplot+[dashed] table [x=lambda, y=Weibull-2, col sep=comma]{\dataFigureUpperbound};

		\addlegendentry{Distributions:}
		\addlegendentry{Deterministic}
		\addlegendentry{Erlang2}
		\addlegendentry{Exponential}
		\addlegendentry{Weibull-$1$}
		\addlegendentry{Weibull-$2$}
	\end{axis}
\end{tikzpicture}}
    \resizebox{0.49\textwidth}{!}{\newcommand{\dataFigure}{TikZFigures/Redundancy_PS_independent_N4_d3.csv}
\newcommand{\dataFigureLowerbound}{TikZFigures/Redundancy_PS_independent_N4_d3_lowerbound.csv}
\newcommand{\dataFigureUpperbound}{TikZFigures/Redundancy_PS_independent_N4_d3_upperbound.csv}

\begin{tikzpicture}
		\begin{axis}[
			xlabel=$\lambda$,
			ylabel=Expected latency,
			ymin=0,
			ymax=5,
			xmin=0,
			xmax=3,
			no markers,
			legend pos= outer north east,
			legend cell align=left]
		\addlegendimage{empty legend}
		\addplot+ table [x=lambda, y=Deterministic, col sep=comma]{\dataFigure};
		\addplot+ table [x=lambda, y=Erlang2, col sep=comma]{\dataFigure};
		\addplot+ table [x=lambda, y=Exponential, col sep=comma]{\dataFigure};
		\pgfplotsset{cycle list shift=1}
		\addplot+ table [x=lambda, y=Weibull-1, col sep=comma]{\dataFigure};
		\addplot+ table [x=lambda, y=Weibull-2, col sep=comma]{\dataFigure};
		
		 \pgfplotsset{cycle list shift=-5}
		\addplot+[dashed] table [x=lambda, y=Deterministic, col sep=comma]{\dataFigureLowerbound};
		\addplot+[dashed] table [x=lambda, y=Erlang2, col sep=comma]{\dataFigureLowerbound};
		\addplot+[dashed] table [x=lambda, y=Exponential, col sep=comma]{\dataFigureLowerbound};
		\pgfplotsset{cycle list shift=-4}
		\addplot+[dashed] table [x=lambda, y=Weibull-1, col sep=comma]{\dataFigureLowerbound};
		\addplot+[dashed] table [x=lambda, y=Weibull-2, col sep=comma]{\dataFigureLowerbound};
		
		\pgfplotsset{cycle list shift=-10}
		\addplot+[dashed] table [x=lambda, y=Deterministic, col sep=comma]{\dataFigureUpperbound};
		\addplot+[dashed] table [x=lambda, y=Erlang2, col sep=comma]{\dataFigureUpperbound};
		\addplot+[dashed] table [x=lambda, y=Exponential, col sep=comma]{\dataFigureUpperbound};
		\pgfplotsset{cycle list shift=-9}
		\addplot+[dashed] table [x=lambda, y=Weibull-1, col sep=comma]{\dataFigureUpperbound};
		\addplot+[dashed] table [x=lambda, y=Weibull-2, col sep=comma]{\dataFigureUpperbound};

		\addlegendentry{Distributions:}
		\addlegendentry{Deterministic}
		\addlegendentry{Erlang2}
		\addlegendentry{Exponential}
		\addlegendentry{Weibull-$1$}
		\addlegendentry{Weibull-$2$}
	\end{axis}
\end{tikzpicture}}
    \caption{The expected latency in the original system and in the lower- and upperbound systems (dashed lines) for PS with i.i.d. replicas, $N=4$, $\mathbb{E}[X]=2$, various job size distributions (see Table~\ref{tab: distributions}) and $d=2$ (left) and $d=3$ (right).
    \label{fig: bounds independent replicas}}
\end{figure}

\subsection{Near-insensitivity identical replicas}
\label{subsec: near insensitivity identical replicas}
It is well known that the expected latency in an ordinary $M/G/1/PS$ system is insensitive to the job size distribution given its mean. In this subsection our aim is to show that the expected latency is nearly insensitive to the job size distribution in redundancy-$d$ systems with processor sharing and identical replicas. To support this claim, we consider the system with various job size distributions. We run the simulation $50$ times, where each run consists of $10^{7}$ arrivals. See also \cite{GHBSW-AJSQ} in which the authors observe a similar near-insensitivity for an $M/G/N/JSQ/PS$ system. 

In Figure~\ref{fig: near insensitivity} it can be seen that the $95$\% confidence intervals for the expected latency differ and are not overlapping for the various job size distributions. However, looking at the vertical axis of the figure, we conclude that the expected latency only differs by $0.05$, which is approximately 1.5\%, between the job size distributions with the lowest and highest variance. 

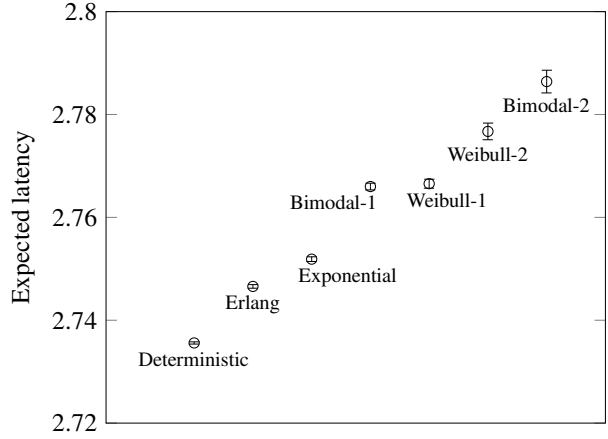
\begin{figure}[htbp]
  \centering
	\resizebox{0.6\textwidth}{!}{\newcommand{\dataFigure}{TikZFigures/Redundancy_PS_identical_N4_ML.csv}

\begin{tikzpicture}
\begin{axis}[
	ymin=2.72,
	ymax=2.8,
	xmin=-0.5,
	xmax=8,
	ylabel=Expected latency,
	xtick=\empty
]

\addplot+ [black, only marks, mark=o,error bars/.cd,y explicit,y dir=both] 
    table [col sep=comma,x=x,y=0.5m,
          y error plus expr=\thisrow{0.5u}-\thisrow{0.5m},
          y error minus expr=\thisrow{0.5m}-\thisrow{0.5l}]{\dataFigure}
node[pos=0.0, below]{\footnotesize Deterministic}
node[pos=0.125, below]{\footnotesize Erlang}
node[pos=0.25, below, xshift=0.5cm]{\footnotesize Exponential}
node[pos=0.5, below, xshift=-0.5cm]{\footnotesize Bimodal-1}  
node[pos=0.625, below, xshift=0.25cm]{\footnotesize Weibull-1} 
node[pos=0.75, below, yshift=-0.1cm]{\footnotesize Weibull-2} 
node[pos=1.0, below, yshift=-0.1cm]{\footnotesize Bimodal-2}; 

\end{axis}
\end{tikzpicture}}
    \caption{Near-insensitivity results for PS with identical replicas, $N=4$, $d=2$, $\tilde{\rho}=0.5$ and various job size distributions (see Table~\ref{tab: distributions}).}
    \label{fig: near insensitivity}
\end{figure}

\subsection{Impact of load on delay performance}
\label{subsec: load}
In Section~\ref{sec: stability conditions stochastic system} the necessary and sufficient stability conditions for the PS discipline are conjectured. However, when fixing the load $\tilde{\rho}$, the expected latency still differs when varying $d$ and $N$. In Figure~\ref{fig: load} we fix the load, so for identical replicas the arrival rate is different for each value of $d$ and therefore does not provide a fair comparison regarding the performance. Since $\mathbb{E}[T] \approx \mathbb{E}[X] = 1$ we infer that, especially for identical replicas, even with a relatively high load of $\tilde{\rho}=0.75$ for $d=N/2$ (most of the time) jobs do not experience delay from other jobs. Even for $\tilde{\rho}=0.95$, which would be considered heavy traffic in the classical setting, the expected latency $\mathbb{E}[T]$ is approximately equal to $2$ for values of $d$ between $3$ and $7$, while for $d=1$ and $d=N=10$, i.e., equivalent to the classical setting, the expected latency $\mathbb{E}[T]=20$. Observe that for identical replicas, the expected latency varies significantly more with the value of $d$ than for i.i.d.\ replicas. 

\begin{figure}[htbp]
  \centering
      \resizebox{0.49\textwidth}{!}{\newcommand{\dataFigure}{TikZFigures/Redundancy_PS_identical_load.csv}

\begin{tikzpicture}
		\begin{axis}[
			xlabel=$d$,
			ylabel=Expected latency,
			ymin=0,
			ymax=10,
			legend pos= outer north east,
			legend cell align=left]
		\addlegendimage{empty legend}
		\addplot+[mark=*, dashed] table [x=d, y=0.5, col sep=comma]{\dataFigure};
		\addplot+[mark=*, dashed] table [x=d, y=0.75, col sep=comma]{\dataFigure};
		\addplot+[mark=*, dashed] table [x=d, y=0.95, col sep=comma]{\dataFigure};
		
		\addlegendentry{Load:}
		\addlegendentry{$\tilde{\rho}=0.5$}
		\addlegendentry{$\tilde{\rho}=0.75$}
		\addlegendentry{$\tilde{\rho}=0.95$}
	\end{axis}
\end{tikzpicture}}
      \resizebox{0.49\textwidth}{!}{\newcommand{\dataFigure}{TikZFigures/Redundancy_PS_independent_load.csv}

\begin{tikzpicture}
		\begin{axis}[
			xlabel=$d$,
			ylabel=Expected latency,
			ymin=0,
			ymax=20,
			legend pos= outer north east,
			legend cell align=left]
		\addlegendimage{empty legend}
		\addplot+[mark=*, dashed] table [x=d, y=0.5, col sep=comma]{\dataFigure};
		\addplot+[mark=*, dashed] table [x=d, y=0.75, col sep=comma]{\dataFigure};
		\addplot+[mark=*, dashed] table [x=d, y=0.95, col sep=comma]{\dataFigure};
		
		\addlegendentry{Load:}
		\addlegendentry{$\tilde{\rho}=0.5$}
		\addlegendentry{$\tilde{\rho}=0.75$}
		\addlegendentry{$\tilde{\rho}=0.95$}
	\end{axis}
\end{tikzpicture}}
    \caption{Simulated expected latency for PS with $N=10$, exponential job sizes $X$ with $\mathbb{E}[X]=1$, various loads and identical replicas (left) and i.i.d.\ replicas (right).}
    \label{fig: load}
\end{figure}
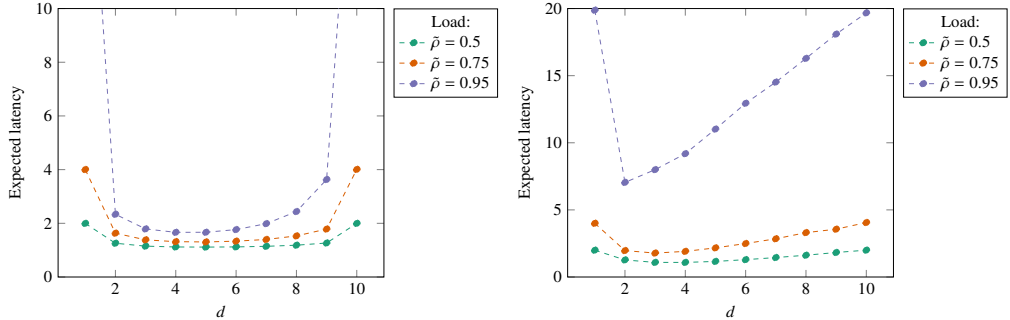

\section{Conclusion}
\label{sec: conclusion}
In this paper we studied the stability condition for redundancy-$d$ systems where all servers follow a processor-sharing discipline. We allow for generally distributed job sizes with possible dependence among the replica sizes of a job being governed by some joint distribution. The stability condition for the associated fluid-limit model is characterized by the expectation of the minimum of $d$ replica sizes being less than the mean interarrival time per server. For identical replicas the stability condition is insensitive to the job size distribution given its mean. Moreover, the stability threshold is inversely proportional to the number of replicas. Thus, in this case a higher degree of redundancy reduces stability. 
For i.i.d.\ replicas the stability threshold decreases (increases) in the number of replicas for job size distributions that are NBU (NWU).  

A natural topic for further research would be to allow for heterogeneous server speeds. In this scenario the lower- and upperbound systems from Sections~\ref{sec: lowerbound} and~\ref{sec: upperbound} continue to apply. Again, we could consider the fluid limits of these systems. Observe that the heterogeneity of the servers is reflected in the service rate per job and thus in the attained service process. Therefore, Property~\ref{prop: fluid limit} that played a key role in establishing the stability condition of the fluid models is not valid anymore. Moreover, it is no longer true that the fluid limits of the lower- and upperbound systems coincide when starting with the same initial conditions for all the virtual queues. 
Our simulations of the lower- and upperbound systems with heterogeneous server speeds reveal that the stability conditions do not coincide, which means that (most probably) the stability conditions of the fluid models do not coincide either. In addition, we observed that for heterogeneous server speeds the stability condition of the original stochastic system is sensitive, i.e., not only dependent on the expectation of the minimum of $d$ replica sizes.
For the PS discipline, the model of heterogeneous server speeds could be even further generalized to also include various job types, each with their own server speed realizations. One example in such a direction is the S\&X model proposed in \cite{GHBSW-DSSJS}. 

The stability condition for FCFS and ROS in the case of general job size distributions is still an open problem. By simulation we observed that the stability condition for PS gives a reasonable approximation for the stability condition for FCFS. Moreover, in all the simulations, we observed that for i.i.d.\ replicas the stability threshold for PS is smaller (larger) than the stability threshold when the job size distribution is NBU (NWU). Note that this observation is in agreement with Remark~\ref{rem: comparing FCFS with PS scaled bernoulli}. Intuitively, this could be explained since for the PS discipline every replica starts at the same time and the queue lengths are approximately equal. Thus for a job size distribution that is NBU (NWU) the PS discipline maximizes the wastage (improvement) of the server capacity. For the FCFS discipline this is not necessarily the case since not all replicas start at the same time or start service at all. Rigorous proofs of these statements remain as a challenge for further research. 

\section*{Acknowledgments}
\label{sec: acknowlegdements}
The work in this paper is supported by the Netherlands Organisation for Scientific Research (NWO) through Gravitation grant NETWORKS 024.002.003.
The authors gratefully acknowledge insightful discussions with Bert Zwart. 

%% The Appendices part is started with the command \appendix;
%% appendix sections are then done as normal sections
%% \appendix

%% If you have bibdatabase file and want bibtex to generate the
%% bibitems, please use
%%

%% else use the following coding to input the bibitems directly in the
%% TeX file.

\appendix
\section{Appendix}

\subsection{Proof fluid limits}
\label{app sec: proof fluid limits}
The proof of Theorem~\ref{thm: fluid limit} is very similar to the proof of Theorem~5.2.1 in \cite{E-STTPS}. The latter proof only relies on the property that $\overline{\eta}^{n}_{i}(s,t)$ is decreasing in $s$ and $\overline{\eta}(\cdot,t)$ is continuous on $[\xi_{i}(t)+\epsilon,t]$, where $\xi_{i}(t) = \sup(u \in [0,t] : \overline{Q}_{i}(u) = 0)$. In order to keep the paper self-contained, we give here the proof.

Applying the fluid scaling to each term in Equation~\eqref{eq: virtual queue length process} gives
\begin{align*}
\overline{Q}_{i}^{n}(t) = \frac{1}{n}\sum_{l=1}^{n \overline{Q}_{l}^{n}(0)} \mathbbm{1}_{\{v'_{il} > \overline{\eta}_{i}^{n}(0,t)\}} + \frac{1}{n}\sum_{k=1}^{A_{i}^{n}(t)} \mathbbm{1}_{\{v_{ik} > \overline{\eta}_{i}^{n}(U_{ik}^{n},t)\}} := I_{i}^{n} + J_{i}^{n}.
\end{align*}
We now proceed to derive $\lim_{n \rightarrow \infty} \overline{Q}_{i}^{n}(t)$, and distinguish two cases depending on whether $\overline{Q}_{i}(u) > 0$ for all $u \in [0,t]$ or not.
Let $\xi_{i}(t) = \sup(u \in [0,t] : \overline{Q}_{i}(u)=0)$. It is useful to distinguish two further cases, depending on whether $\xi_{i}(t) < t$ or $\xi_{i}(t) = t$. We start with the former case, and fix $\epsilon > 0$ such that $\xi_{i}(t) + \epsilon < t$. 
Now 
\begin{align*}
J_{i}^{n} &= \frac{1}{n}\sum_{k=1}^{A_{i}^{n}(\xi_{i}(t) + \epsilon)} \mathbbm{1}_{\{v_{ik} > \overline{\eta}_{i}^{n}(U_{ik}^{n},t)\}} \\
&+ \frac{1}{n}\sum_{k=A_{i}^{n}(\xi_{i}(t) + \epsilon)+1}^{A_{i}^{n}(t)} \mathbbm{1}_{\{v_{ik} > \overline{\eta}_{i}^{n}(U_{ik}^{n},t)\}} := J_{1,i}^{n} + J_{2,i}^{n}.
\end{align*}
We first determine $\lim_{n \rightarrow \infty} J_{2,i}^{n}$. By definition $\overline{Q}_{i}(u) > 0$ for all $[\xi_{i}(t) + \epsilon, t]$. Hence, the bounded convergence theorem yields
\begin{align*}
\lim_{n \rightarrow \infty} \overline{\eta}_{i}^{n}(u,v) = \overline{\eta}_{i}(u,v),
\end{align*}
for all $t \geq v \geq u \geq \xi_{i}(t) + \epsilon$. Since $\overline{\eta}_{i}^{n}(s,t)$ is decreasing in $s$ and $\overline{\eta}_{i}(\cdot,t)$ is continuous on $[\xi_{i}(t) + \epsilon, t]$, the convergence is uniform on $[\xi_{i}(t) + \epsilon, t]$, i.e., for any $\delta>0$ there exists an $n_{\delta}$ such that
\begin{align}
\label{app eq: uniform convergence eta}
\sup_{s \in [\xi_{i}(t) + \epsilon, t]} |\overline{\eta}_{i}^{n}(s,t) - \overline{\eta}_{i}(s,t)| \leq \delta, ~~~ \text{ for all } n \geq n_{\delta}.
\end{align}

We partition the interval $[\xi_{i}(t) + \epsilon, t]$ into $N_{1}$ subintervals $[t_{j-1}^{N_{1}},t_{j}^{N_{1}}]$, $j=1,\dots,N_{1}$, for some integer $N_{1} \geq 1$, in such a way that\\ $\max_{j=0,\dots,N_{1}}(t_{j}^{N_{1}}-t_{j-1}^{N_{1}}) \rightarrow 0$ as $N_{1} \rightarrow \infty$.
Then,
\begin{align*}
J_{2,i}^{n} = \frac{1}{n} \sum_{j=1}^{N_{1}} \sum_{k=A_{i}^{n}(t_{j}^{N_{1}})+1}^{A_{i}^{n}(t_{j+1}^{N_{1}})} \mathbbm{1}_{\{v_{ik} > \overline{\eta}_{i}^{n}(U_{ik}^{n},t)\}}.
\end{align*}
Suppose that $t_{j-1}^{N_{1}} \leq U_{ik}^{n} \leq t_{j}^{N_{1}}$ for some $j \in \{1,\dots,N_{1}\}$, some $k \in \{A_{i}^{n}(\xi_{i}(t) + \epsilon)+1,\dots,A_{i}^{n}(t)\}$, and some $n > n_{\delta}$. It then follows from \eqref{app eq: uniform convergence eta} that for $n > n_{\delta}$
\begin{align*}
\overline{\eta}_{i}(t_{j}^{N_{1}},t) - \delta \leq \overline{\eta}_{i}^{n}(U_{ik}^{n},t) \leq \overline{\eta}_{i}(t_{j-1}^{N_{1}},t) + \delta,
\end{align*}
which yields
\begin{align*}
&\frac{1}{n} \sum_{j=1}^{N_{1}} \sum_{k=A_{i}^{n}(t_{j-1}^{N_{1}})+1}^{A_{i}^{n}(t_{j}^{N_{1}})} \mathbbm{1}_{\{v_{ik} > \overline{\eta}_{i}(t_{j}^{N_{1}},t) - \delta \}} \leq J_{2,i}^{n} \\
&\leq \frac{1}{n} \sum_{j=1}^{N_{1}} \sum_{k=A_{i}^{n}(t_{j-1}^{N_{1}})+1}^{A_{i}^{n}(t_{1}^{N_{1}})} \mathbbm{1}_{\{v_{ik} > \overline{\eta}_{i}(t_{j-1}^{N_{1}},t) + \delta\}}.
\end{align*}
Using Lemma~$5.1$ in \cite{GRZ-FLPCI}, we obtain
\begin{align*}
\limsup_{n \rightarrow \infty} J_{2,i}^{n} \leq \frac{\lambda}{{N \choose d}} \sum_{j=1}^{N_{1}} (t_{j}^{N_{1}}-t_{j-1}^{N_{1}}) \mathbb{P}(v_{ik} > \overline{\eta}_{i}(t_{j-1}^{N_{1}},t) + \delta),
\end{align*}
\begin{align*}
\liminf_{n \rightarrow \infty} J_{2,i}^{n} \geq \frac{\lambda}{{N \choose d}} \sum_{j=1}^{N_{1}} (t_{j}^{N_{1}}-t_{j-1}^{N_{1}}) \mathbb{P}(v_{ik} > \overline{\eta}_{i}(t_{j}^{N_{1}},t) - \delta).
\end{align*}
For $s \in [\xi_{i}(t) + \epsilon, t]$ the bounded convergence theorem implies that
\begin{align*}
\lim_{N_{1} \rightarrow \infty} \sum_{j=1}^{N_{1}} \mathbbm{1}_{[t_{j}^{N_{1}}-t_{j-1}^{N_{1}})}(s) \mathbb{P}(v_{ik} > \overline{\eta}_{i}(t_{j-1}^{N_{1}},t) + \delta) = \mathbb{P}(v_{ik} > \overline{\eta}_{i}(s,t) + \delta),
\end{align*}
\begin{align*}
\lim_{N_{1} \rightarrow \infty} \sum_{j=1}^{N_{1}} \mathbbm{1}_{[t_{j}^{N_{1}}-t_{j-1}^{N_{1}})}(s) \mathbb{P}(v_{ik} > \overline{\eta}_{i}^{n}(t_{j}^{N_{1}},t) - \delta) = \mathbb{P}(v_{ik} > \overline{\eta}_{i}(s,t) - \delta).
\end{align*}
Letting $N_{1} \rightarrow \infty$, we deduce
\begin{align*}
\limsup_{n \rightarrow \infty} J_{2,i}^{n} \leq \frac{\lambda}{{N \choose d}} \int_{\xi_{i}(t) + \epsilon}^{t} \mathbb{P}(v_{ik} > \overline{\eta}_{i}(s,t) + \delta) \text{d}s,
\end{align*}
\begin{align*}
\liminf_{n \rightarrow \infty} J_{2,i}^{n} \geq \frac{\lambda}{{N \choose d}} \int_{\xi_{i}(t) + \epsilon}^{t} \mathbb{P}(v_{ik} > \overline{\eta}_{i}(s,t) - \delta) \text{d}s.
\end{align*}
Passing $\delta \downarrow 0$ and $\epsilon \downarrow 0$, we obtain because of continuity,
\begin{align}
\label{app eq: J2 term}
\lim_{n \rightarrow \infty} J_{2,i}^{n} = \frac{\lambda}{{N \choose d}} \int_{\xi_{i}(t)}^{t} \mathbb{P}(v_{ik} > \overline{\eta}_{i}(s,t)) \text{d}s.
\end{align}
We now determine $\lim_{n \rightarrow \infty} J_{1,i}^{n}$ and $\lim_{n \rightarrow \infty} I_{i}^{n}$. Fatou's lemma and the fact that $\xi_{i}(t) < t$ imply
\begin{align}
\label{app eq: lim inf eta}
\liminf_{n \rightarrow \infty} \overline{\eta}_{i}^{n}(0,t) \geq \int_{u=0}^{t} \liminf_{n \rightarrow \infty} \varphi(\overline{Q}_{i}^{n}(u)) \text{d}u =  \overline{\eta}_{i}(0,t) = \infty.
\end{align}
We partition the interval $[0,\xi_{i}(t)]$ into $N_{2}$ subintervals $[s_{j-1}^{N_{2}},s_{j}^{N_{2}}]$, $j=1,\dots,N_{2}$, for some integer $N_{2} \geq 1$, in such a way that $\max_{j=1,\dots,N_{2}} (s_{j}^{N_{2}}-s_{j-1}^{N_{2}}) \rightarrow 0$ as $N_{2} \rightarrow \infty$. Suppose $s_{j-1}^{N_{2}} \leq U_{ik}^{n} \leq s_{j}^{N_{2}}$ for some $j \in \{1,\dots,N_{2}\}$ and $k \in \{1,\dots,A_{i}^{n}(\xi_{i}(t))\}$. It then follows from \eqref{app eq: lim inf eta} that
\begin{align}
\label{app eq: lim inf partition eta}
\liminf_{n \rightarrow \infty} \overline{\eta}_{i}^{n}(U_{ik}^{n},t) \geq \liminf_{n \rightarrow \infty} \overline{\eta}_{i}^{n}(s_{j}^{N_{2}},t) \geq  \overline{\eta}_{i}(s_{j}^{N_{2}},t) = \infty.
\end{align}
For $J_{1,i}^{n}$ we have
\begin{align*}
0 \leq J_{1,i}^{n} \leq \frac{1}{n}\sum_{k=1}^{A_{i}^{n}(\xi_{i}(t))} \mathbbm{1}_{\{v_{ik} > \overline{\eta}_{i}^{n}(U_{ik}^{n},t)\}} + \frac{1}{n}(A_{i}^{n}(\xi_{i}(t) + \epsilon) - A_{i}^{n}(\xi_{i}(t))).
\end{align*}
The first term on the right-hand side tends to $0$ by \eqref{app eq: lim inf partition eta}, while the second term converges to $\frac{\lambda}{{N \choose d}} \epsilon$ according to Lemma~$5.1$ in \cite{GRZ-FLPCI}. Passing $\epsilon \downarrow 0$, we obtain
\begin{align}
\label{app eq: J1 term}
\lim_{n \rightarrow \infty} J_{1,i}^{n} = 0 = \frac{\lambda}{{N \choose d}} \int_{s=0}^{\xi_{i}(t)} \mathbb{P}(v_{ik} > \overline{\eta}_{i}(0,t))\text{d}s.
\end{align}
The term $I_{i}^{n}$ follows from \eqref{app eq: lim inf eta}:
\begin{align}
\label{app eq: term inital jobs}
\lim_{n \rightarrow \infty} \frac{1}{n} \sum_{l=1}^{n \overline{Q}_{i}^{n}(0)} \mathbbm{1}_{\{v'_{il} > \overline{\eta}_{i}^{n}(0,t)\}} = 0 = \overline{Q}_{i}(0) \mathbb{P}(v'_{il} > \overline{\eta}_{i}(0,t)).
\end{align}
Taking the sum of \eqref{app eq: J1 term}, \eqref{app eq: J2 term} and \eqref{app eq: term inital jobs}, yields the right-hand side of \eqref{eq: fluid limit}. The limit on the left-hand side is $\lim_{n \rightarrow \infty} \overline{Q}_{i,n}(t) = \overline{Q}_{i}(t)$. This proves \eqref{eq: fluid limit} in case $\xi_{i}(t) < t$. 

In case $\xi_{i}(t) = t$, Equation \eqref{eq: fluid limit} immediately follows from the fact that $\overline{\eta}_{i}^{n}(s,t) \rightarrow \infty$ for any $s \in [0,t]$.

It remains to treat the case when $\overline{Q}_{i}(u) > 0$ for all $u \in [0,t]$. Then $\overline{\eta}_{i}^{n}(u,v)$ converges uniformly to $\overline{\eta}_{i}(u,v)$ for any $u,v \in [0,t]$, while the expression $J_{i}^{n}$ follows by the same argument as used in $J_{2,i}^{n}$, on the interval $[\xi_{i}(t) + \epsilon, t]$. For any $\epsilon > 0 $, there exists an $n_{\epsilon}$ such that $\overline{\eta}_{i}^{n}(0) \in (\overline{\eta}_{i}(0,t) - \epsilon, \overline{\eta}_{i}(0,t) + \epsilon)$ for all $n > n_{\epsilon}$. Multiplying and dividing $I_{i}^{n}$ by $\overline{Q}_{i,n}(0)$, we deduce
\begin{align*}
\limsup_{n \rightarrow \infty} \frac{1}{n} \sum_{l=1}^{n \overline{Q}_{i}^{n}(0)} \mathbbm{1}_{\{v'_{il} > \overline{\eta}_{i}^{n}(0)\}} \geq \overline{Q}_{i}(0) \mathbb{P}(v'_{il} > \overline{\eta}_{i}(0,t) - \epsilon),
\end{align*}
\begin{align*}
\liminf_{n \rightarrow \infty} \frac{1}{n} \sum_{l=1}^{n \overline{Q}_{i}^{n}(0)} \mathbbm{1}_{\{v'_{il} > \overline{\eta}_{i}^{n}(0)\}} \leq \overline{Q}_{i}(0) \mathbb{P}(v'_{il} > \overline{\eta}_{i}(0,t) + \epsilon).
\end{align*}
Letting $\epsilon \downarrow 0$, we find that
\begin{align*}
\lim_{n \rightarrow \infty} I_{i}^{n} = \overline{Q}_{i}(0) \mathbb{P}(v'_{il} > \overline{\eta}_{i}(0,t)).
\end{align*}
This completes the proof.

\end{document}